\documentclass[11pt, leqno]{amsart}
\usepackage{amsmath}
\usepackage{amsfonts}
\usepackage{amssymb}
\usepackage{dsfont}
\usepackage{graphicx}
\usepackage{mathrsfs}
\usepackage[margin=2.90cm]{geometry}
\usepackage{url}
\usepackage{hyperref}
\providecommand{\U}[1]{\protect\rule{.1in}{.1in}}
\newtheorem{theorem}{Theorem}[section]
\newtheorem*{acknowledgement*}{Acknowledgements}

\newtheorem{lemma}[theorem]{Lemma}

\newtheorem{proposition}[theorem]{Proposition}
\newtheorem{remark}[theorem]{Remark}

\newtheorem{theoremA}{Theorem} % alphabetic numbering
\newtheorem{corollaryA}[theoremA]{Corollary}

\def\<{\left\langle}
\def\>{\right\rangle}

\newcommand{\matrice}{\begin{pmatrix}}
\newcommand{\ok}{\end{pmatrix}}
\newcommand{\dmatrice}{\begin{vmatrix}}
\newcommand{\dok}{\end{vmatrix}}

\newcommand{\real}[1]{{\bf R}^{#1}}

\def\<{\left\langle}
\def\>{\right\rangle}

\begin{document}
\title[Quantitative index bounds for translators via topology]{Quantitative index bounds for translators via topology}


\author[Debora Impera]{Debora Impera}
\address[Debora Impera]{Dipartimento di Scienze Matematiche "Giuseppe Luigi Lagrange", Politecnico di Torino, Corso Duca degli Abruzzi, 24, Torino, Italy, I-10129}
\email{debora.impera@gmail.com}

\author[Michele Rimoldi]{Michele Rimoldi}
\address[Michele Rimoldi]{Dipartimento di Scienze Matematiche "Giuseppe Luigi Lagrange", Politecnico di Torino, Corso Duca degli Abruzzi, 24, Torino, Italy, I-10129}
\email{michele.rimoldi@gmail.com}

%\date{\today}

\begin{abstract}
We obtain a quantitative estimate on the generalised index of translators for the mean curvature flow with bounded norm of the second fundamental form. The estimate involves the dimension of the space of weighted square integrable $f$-harmonic $1$-forms. By the adaptation to the weighted setting of Li-Tam theory developed in previous works, this yields estimates in terms of the number of ends of the hypersurface when this is contained in a upper halfspace with respect to the translating direction. When there exists a point where all principal curvatures are distinct we estimate the nullity of the stability operator. This permits to obtain quantitative estimates on the stability index via the topology of translators with bounded norm of the second fundamental form which are either two-dimensional or (in higher dimension) have finite topological type and are contained in a upper halfspace. 
\end{abstract}

\maketitle
\tableofcontents

\section{Introduction and main results}
An isometrically immersed complete (orientable) hypersurface of the Euclidean space $x:\Sigma^m\to(\real{m+1}, \<\,,\,\>)$ is said to be a translator of the mean curvature flow if its mean curvature vector field satisfies the equation
\begin{equation}\label{T}
\mathbf{H}=\bar{V}^{\bot}
\end{equation}
for some parallel unit length vector field $\bar{V}$ in $\real{m+1}$, where $(\cdot)^{\bot}$ denotes the projection on the normal bundle of $\Sigma$.  The importance of translators comes from the fact that they generate translating solutions of the mean curvature flow and these, in turn, model the formation of Type II singularities when starting from an initial mean convex closed hypersurface. It is by now well-known that equation \eqref{T} turns out to be the Euler-Lagrange equation for the weighted volume functional 
\[
\ \mathrm{vol}_{f}(\Sigma)=\int_{\Sigma}e^{-f}d\mathrm{vol}_{\Sigma},
\]
when choosing $f=-\<x,\bar{V}\>$.

Given a translator and a compactly supported normal variation $u\nu$, the second variation formula for the weighted volume functional yields the quadratic form
\[
\ Q_{f}(u,u)=\int_{\Sigma}\left(|\nabla u|^2-|A|^{2}u^{2}\right)e^{-f}d\mathrm{vol}_{\Sigma}.
\]
The associated stability operator of $\Sigma$ is then given by
\[
\ L_{f}=\Delta_{f}-|A|^{2},
\]
where $\Delta_{f}=\Delta+\<\nabla f, \nabla\>$, and $\Delta\doteq-\mathrm{div}(\nabla)$. The $f$-index of $\Sigma$ is defined in terms of the generalized Morse index of $L_{f}$ on $\Sigma$. Namely, given a relatively compact domain $\Omega\Subset\Sigma$ we define
\[
\ \mathrm{Ind}^{L_{f}}(\Omega)=\sharp\left\{\mathrm{negative\,eigenvalues\,of\,}L_{f}\,\mathrm{on\,} C_{0}^{\infty}(\Omega)\right\}.
\]
The $f$-index of $\Sigma$ is then defined as
\[
\ \mathrm{Ind}_{f}(\Sigma)\doteq\mathrm{Ind}^{L_{f}}(\Sigma)=\sup_{\Omega\subset\subset\Sigma}\mathrm{Ind}^{L_{f}}(\Omega).
\]

We denote by $\mathcal{H}^{1}_{f}(\Sigma)$ the space of $f$-harmonic one-forms which are square integrable with respect to the weighted measure:
\[
\ \mathcal{H}^{1}_{f}(\Sigma)\doteq\left\{\omega\in\Lambda^{1}T^{*}\Sigma\,:\,d\omega=\delta_{f}\omega=0,\,\,\int_{\Sigma}|\omega|^{2}e^{-f}d\mathrm{vol}_{\Sigma}<+\infty\right\},
\]
where $\delta_{f}=\delta+i_{\nabla f}$ and $\delta$ is the usual codifferential. If $\mathrm{Ind}_{f}(\Sigma)<\infty$ we note that it is a consequence of \cite[Theorem 3]{IR_fMin} (see also the considerations before Corollary 1 in \cite{IR_fMin}) that $\mathrm{dim}\left(\mathcal{H}^{1}_{f}(\Sigma)\right)<\infty$.

In the previous work \cite{ImperaRimoldi_Transl} (see also Appendix \ref{AppA} for some corrections and comments about \cite{ImperaRimoldi_Transl}) we highlighted how the realm of weighted manifolds and $f$-minimal hypersurfaces can naturally give strong enough results about the topology at infinity (namely about the number of ends) for translators which are $f$-stable or have finite $f$-index. In particular one can prove that, if $m\geq 3$, $f$-stable translators have at most one end, and translators with finite $f$-index have finitely many ends, provided that they are contained in a upper halfspace with respect to the translating direction. However this latter result, which was proved through the adaptation to the weighted setting of Li-Tam theory, has only a qualitative nature.

Inspired by the recent work by C. Li in the setting of minimal hypersurfaces in the Euclidean space, \cite{ChaoLi}, in this paper we will combine the weighted Li-Tam theory discussed above and a technique pioneered by A. Savo and A. Ros (\cite{Savo}, \cite{Ros}) to prove the following quantitative result.

\begin{theoremA}\label{ThIndPlusNullEst}
Let $x:\Sigma^m\to\real{m+1}$ be a translator with $|A|\in L^{\infty}(\Sigma)$ and $\mathrm{Ind}_{f}(\Sigma)<+\infty$. Then
\begin{equation}\label{IndEst}
\mathrm{Ind}_{f}(\Sigma)+\mathrm{Null}_{f}(\Sigma)\geq\frac{2}{m(m+1)}\mathrm{dim}(\mathcal{H}^{1}_{f}(\Sigma)),
\end{equation}
where $\mathrm{Null}_{f}(\Sigma)$ is the dimension of the space of Jacobi functions which are square integrable with respect to the weighted measure. In particular, if $m\geq 3$ and $x(\Sigma)$ is contained in a upper halfspace with respect to the translating direction, namely 
\[
\ x(\Sigma)\subset\Pi_{\bar{V},a}=\left\{p\in\mathbb{R}^{m+1}:\,\langle\bar{V},p\rangle\geq a\right\}\]
for some $a\in\mathbb{R}$, then
\begin{equation}\label{IndEst2}
\mathrm{Ind}_{f}(\Sigma)+\mathrm{Null}_{f}(\Sigma)\geq\frac{2}{m(m+1)}(\sharp \{ends\}-1)
\end{equation}
\end{theoremA}

\begin{remark}\label{Rem1.1}
\rm{In \eqref{IndEst2} we are using the fact that, since in our assumptions every end of a translator is non $f$-parabolic, 
\[
\ \mathrm{dim}(\mathcal{H}^{1}_{f}(\Sigma))\geq \sharp\left\{\mathrm{ends}\right\}-1;
\]
see \cite{ImperaRimoldi_Transl} and Appendix \ref{AppA} for more details.}
\end{remark}

\begin{remark}
\rm{Note that, up to our knowledge, all known examples of translators which can be found in literature satisfy the condition $|A|\in L^{\infty}(\Sigma)$; besides classical examples see e.g. \cite{N1}, \cite{N2}, and \cite{DDN}.}
\end{remark}

When there exists a point where all principal curvatures are distinct it is actually possible to estimate the nullity of the stability operator. This permits to get the following index estimate.

\begin{theoremA}\label{ThIndEst}
Let $x:\Sigma^m\to\real{m+1}$ be a translator with $|A|\in L^{\infty}(\Sigma)$ and $\mathrm{Ind}_{f}(\Sigma)<+\infty$. If there exists a point $p$ on $\Sigma$ where all the principal curvatures are distinct, then
\begin{equation*}\label{IndEstDistCurv}
\mathrm{Ind}_{f}(\Sigma)\geq\frac{2}{m(m+1)}\left(\mathrm{dim}(\mathcal{H}^{1}_{f}(\Sigma))-2m+1\right).
\end{equation*}
In particular, if $m\geq 3$ and $x(\Sigma)$ is contained in a upper halfspace then
\begin{equation*}\label{IndEstDistCurv1}
\mathrm{Ind}_{f}(\Sigma)\geq\frac{2}{m(m+1)}(\sharp\{ends\}-2m).
\end{equation*}
\end{theoremA}
Note that, in dimension $2$, either $\Sigma$ is totally umbilical or it admits a point where all the principal curvatures are distinct. Since the only totally umbilical translators are the planes containing the translating direction $\bar{E}_3$ (which are $f$-stable), in the two-dimensional case Theorem \ref{ThIndEst} gives the following
\begin{corollaryA}\label{CoroIndEst2D}
Let $x:\Sigma^2\to\real{3}$ be a translator with $|A|\in L^{\infty}(\Sigma)$ and $0\neq\mathrm{Ind}_{f}(\Sigma)<+\infty$. Then 
\begin{equation*}\label{IndEstDistCurv2D}
\mathrm{Ind}_{f}(\Sigma)\geq\frac{1}{3}\left(\mathrm{dim}(\mathcal{H}^{1}_{f}(\Sigma))-3\right).
\end{equation*}
\end{corollaryA}
Adapting a construction by H. M. Farkas and I. Kra, \cite{FK}, in \cite[Theorem D]{IRS_IndexBettiSS} it is proved that for any $2$-dimensional orientable connected complete surface $\Sigma$, and any $f\in C^{\infty}(\Sigma)$, 
\[
\ \mathrm{dim}\mathcal{H}^{1}_{f}(\Sigma)\geq 2g,
\]
where $g$ is the genus of $\Sigma$. Using Corollary \ref{CoroIndEst2D}, this in particular yields the following effective estimate in the two-dimensional case.

\begin{theoremA}\label{Genus2D}
Let $x:\Sigma^2\to\real{3}$ be a translator with $|A|\in L^{\infty}(\Sigma)$ and $\mathrm{Ind}_{f}(\Sigma)<\infty$. Then
\[
\ \mathrm{Ind}_{f}(\Sigma)\geq \frac{2}{3}g-1.
\]
\end{theoremA}

\begin{remark}
\rm{As a consequence of Theorem \ref{Genus2D} we can conclude that any $f$-stable translator with $|A|\in L^{\infty}$ has at most genus $1$. We remark that this last fact was indeed independently improved in the very recent preprint \cite{KS} (which appeared on the Arxiv preprint server while we were reviewing a final version of this work). In that paper it is actually obtained that every $f$-stable translator has genus $0$. The proof in \cite{KS} relies on a general proposition due to M. Gaffney, \cite{Gaf}, as well as on an adaptation of a computation in \cite{P}. Note however that the result in \cite{KS} only concerns $f$-stable translators, while the main focus in our Theorem \ref{Genus2D} is to relate quantitatively the $f$-index and the genus.
}
\end{remark}

As we shall see in Section \ref{SecRelfHarmTop} in presence of a weighted $L^{2}$-Sobolev inequality, one can estimate the dimension of $\mathcal{H}^{1}_{f}$ by means of the dimension of the first cohomology group with compact support on $\Sigma$. By Lemma \ref{4.2} in Appendix \ref{AppA} and Theorem \ref{ThIndPlusNullEst} this result permits to obtain the following estimate when the translator has finite topological type.

\begin{theoremA}\label{ThFiniteTop}
Let $x:\Sigma^{m\geq 3}\to\real{m+1}$ be a translator with finite topological type and $|A|\in L^{\infty}(\Sigma)$. Assume that $x(\Sigma)$ is contained in a upper halfspace. Then $\Sigma$ is diffeomomorphic to a compact hypersurface $\bar{\Sigma}\to\real{m+1}$ with a finite number of points removed $\left\{p_{1}, \,\ldots,\,p_{r}\right\}$, and
\[
\ \mathrm{Ind}_{f}(\Sigma)+\mathrm{Null}_{f}(\Sigma)\geq \frac{2}{m(m+1)}\left(b_{1}(\bar{\Sigma})+r-1\right),
\]
where $b_{1}(\bar{\Sigma})$ is the first Betti number of the compactification $\bar{\Sigma}$ of $\Sigma$.
\end{theoremA}

The paper is organized as follows. In Section \ref{SecBasEq} we collect some basic equations that we shall use in the proofs of our results. Section \ref{SecIndPlusNull} is devoted to the proof of the index plus nullity estimate, i.e. Theorem \ref{ThIndPlusNullEst}. Section \ref{EstNull} concerns with the case in which the hypersurface admits a point where all principal curvature are distinct, and hence take care of the estimate of the nullity of the stability operator which yields Theorem \ref{ThIndEst} and Corollary \ref{CoroIndEst2D}. We end up the paper with Section \ref{SecRelfHarmTop} where we discuss the relation between $f$-harmonic $1$- forms and topology on a weighted manifold admitting a weighted $L^{2}$-Sobolev inequality and we eventually deduce the validity of Theorem \ref{ThFiniteTop}. In Appendix \ref{AppA} we provide a correct proof of Lemma 4.2 in \cite{ImperaRimoldi_Transl} and make some comments about \cite{ImperaRimoldi_Transl}.
%%%%%%%%%%%%%%%%%%%%%%%%%%%%%%%%%%%%%%%%%%%%%%%%
\section{Basic equations}\label{SecBasEq}
Let $x:\Sigma^m\to\real{m+1}$ be a translator of the mean curvature flow. Letting $\left\{\bar{E}_{1},\ldots,\bar{E}_{m+1}\right\}$ be the standard orthonormal basis of $\real{m+1}$, we assume from now on, without loss of generality, that $\bar{V}=\bar{E}_{m+1}$ in \eqref{T}. Furthermore we will set $\eta_{i}=E_{i}^{\flat}$, $E_{i}$ being the projections of $\bar{E}_{i}$ on $\Sigma$.

\begin{lemma}\label{LemBasEq}
Set $f=-\<x,\bar{E}_{m+1}\>$ and denote by $\nu$ the unit normal vector to $\Sigma$ and by $A$ the second fundamental form of the immersion. Let $\omega\in\Lambda^{1}T^{*}\Sigma$, $\xi=\omega^{\sharp}$ and set $X_{ij}\doteq\<\bar{E}_{i},\nu\>E_{j}-\<\bar{E}_{j},\nu\>E_{i}$. Then
\begin{align}
\nabla_{\xi}E_{i}=&\<\bar{E}_{i}, \nu\>A\xi;\label{1}\\
\nabla\<\bar{E}_{i}, \nu\>=&-AE_{i};\label{2}\\
\Delta_{f}\<\bar{E_{i}},\nu\>=&\<\bar{E}_{i}, \nu\>|A|^2;\label{3}\\
\Delta_{f}\<x,\bar{E}_{i}\>=&-\<\bar{E}_{i},\bar{E}_{m+1}\>;\label{4}\\
\Delta_{f}\<E_{i},\xi\>=&2\<AE_{i},A\xi\>+\<\Delta_{f}^{[1]}\omega,\eta_{i}\>\label{5}
\\&-2\<\bar{E}_{i},\nu\>\sum_{k}\<Ae_{k},\nabla\xi(e_{k})\>;\nonumber\\
\Delta_{f}\left(\<\bar{E}_{j},\nu\>\<E_{i},\xi\>\right)=&|A|^2\<\bar{E}_{i},\nu\>\<E_{i},\xi\>+2\<\nabla\<E_{i},\xi\>,AE_{j}\>\label{6}
\\&+2\<AE_{i},A\xi\>\<\bar{E}_{j},\nu\>
+\<\Delta_{f}^{[1]}\omega,\eta_{i}\>\<\bar{E}_{j},\nu\>\nonumber
\\&-2\<\bar{E}_{i},\nu\>\<\bar{E}_{j},\nu\>\sum_{k}\<Ae_{k},\nabla\xi(e_{k})\>;\nonumber\\
\Delta_{f}\<X_{ij},\xi\>=&|A|^2\<X_{ij},\xi\>+2v_{\omega,ij}+\<\Delta_{f}^{[1]}\omega,X_{ij}^{\flat}\>\label{7},
\end{align}
where $\left\{e_{k}\right\}_{k=1}^{m}$ is a local orthonormal frame of $\Sigma$ and $v_{\omega,ij}=\<E_{j},\nabla_{AE_{i}}\xi\>-\<E_{i},\nabla_{AE_{j}}\xi\>$.
%In particular, if $\omega\in L^{2}_{f}\mathcal{H}^{1}(\Sigma_{f})$, $g_{\omega,ij}:=\<X_{ij},\xi\>$ and $v_{\omega,ij}:=\<E_{j},\nabla_{AE_{i}}\xi\>-\<E_{i},\nabla_{AE_{j}}\xi\>$, we get that
%\begin{equation}
%L_{f}g_{\omega,ij}=2v_{\omega,ij}.
%\end{equation}
\end{lemma}
%$\<X_{ij},\xi\>=\<\bar{E}_{i},\nu\>\<E_{j},\xi\>-\<\bar{E}_j,\nu\>\<E_{i},\xi\>$
\begin{proof}
In order to prove \eqref{1}, we compute
\begin{align*}
\nabla_{\xi}E_{i}=&\bar{\nabla}_{\xi}E_{i}-\<\bar{\nabla}_{\xi}E_{i},\nu\>\nu\\
=&-\<\bar{E}_{i}, \nu\>\bar{\nabla}_{\xi}\nu=\<\bar{E}_{i},\nu\>A\xi.
\end{align*}
Letting $\left\{e_{j}\right\}_{j=1}^{m}$ be a local orthonormal frame on $\Sigma$, we have that
\begin{align*}
\nabla\<\bar{E}_{i},\nu\>=&\sum_{j}e_{j}\left(\<\bar{E}_{i},\nu\>\right)e_{j}=\sum_{j}\<\bar{E}_{i},\bar{\nabla}_{e_{j}}\nu\>e_{j}\\
=&-\sum_{j}\<AE_{i},e_{j}\>e_{j}=-AE_{i},
\end{align*}
i.e., equation \eqref{2}.

As for \eqref{3}, letting $Y\in T\Sigma$, we have by Codazzi's equation that
\begin{align*}
\<\nabla_{Y}\nabla\<\bar{E}_{i},\nu\>,Y\>=&-\<\nabla_{Y}AE_{i},Y\>\\
=&-\<(\nabla_{Y}A)E_{i}, Y\>-\<AY,\nabla_{Y}E_{i}\>\\
=&- \<(\nabla_{E_{i}}A)Y, Y\>-\<AY, AY\>\<\bar{E}_{i},\nu\>\\
=&-\<(\nabla_{E_{i}}A)Y, Y\>-\<A^{2}Y,Y\>\<\bar{E}_{i}, \nu\>.
\end{align*}
Taking minus the trace of the previous equation we get
\begin{align*}
\Delta\<\bar{E}_{i},\nu\>=&\<\nabla H, E_{i}\>+\<\bar{E}_{i},\nu\>|A|^2\\
=&-\<AE_{m+1}, E_{i}\>+\<\bar{E}_{i},\nu\>|A|^2\\
=&-\<E_{m+1},AE_{i}\>+\<\bar{E}_{i},\nu\>|A|^2\\
=&-\<\nabla f, \nabla\<\bar{E}_{i},\nu\>\>+\<\bar{E}_{i},\nu\>|A|^2,
\end{align*}
that is \eqref{3}. To prove \eqref{4}, letting $Y\in T\Sigma$, we have that
\[
\ \<Y,\nabla\<x,\bar{E}_{i}\>\>=\<\bar{\nabla}_{Y}x, \bar{E}_{i}\>=\<Y, \bar{E}_{i}\>,
\]
i.e. $\nabla\<x, \bar{E}_{i}\>=E_{i}$. Thus we obtain that for $Y\in T\Sigma$,
\begin{align*}
\<\nabla_{Y}\nabla\<x,\bar{E}_{i}\>,Y\>=&\<\nabla_{Y}E_{i},Y\>=\<\bar{\nabla}_{Y}\left(\bar{E}_{i}-\<\bar{E}_{i},\nu\>\nu\right),Y\>\\
=&-\<\bar{E}_{i},\nu\>\<\bar{\nabla}_{Y}\nu,Y\>=\<\bar{E}_{i},\nu\>\<AY,Y\>.
\end{align*} 
Taking minus the trace of this latter and using \eqref{T}, we obtain that
\[
\ \Delta\<x, \bar{E}_{i}\>=-\<\bar{E}_{i},\nu\>H=-\<\bar{E}_{i},\nu\>\<\bar{E}_{m+1},\nu\>.
\]
Moreover
\[
\ \<\nabla\<x, \bar{E}_{i}\>,\nabla f\>=-\<E_{i},E_{m+1}\>,
\]
hence
\[
\ \Delta_{f}\<x,\bar{E}_{i}\>=-\<\bar{E}_{i},\nu\>\<\bar{E}_{m+1},\nu\>-\<E_{i},E_{m+1}\>=-\<\bar{E}_{i},\bar{E}_{m+1}\>.
\]
Let now $\omega\in\Lambda^{1}T^{*}\Sigma$ and $\xi=\omega^{\sharp}$. We recall that the following Weitzenb\"ock formula holds in the weighted setting:
\begin{equation}\label{W}
\Delta_{f}^{[1]}\omega=\nabla_{f}^{*}\nabla \omega+ \mathrm{Ric}_{f}(\omega^{\sharp}),
\end{equation}
where $\nabla^{*}_{f}=\nabla^{*}+i_{\nabla f}$ and $\mathrm{Ric}_{f}=\mathrm{Ric}+\mathrm{Hess}(f)$; see e.g. \cite{IRS_IndexBettiSS}. Using \eqref{W} we get that
\begin{align*}
\Delta_{f}\<E_{i},\xi\>=&\Delta_{f}\<\eta_{i},\omega\>=\<\nabla^{*}_{f}\nabla\eta_{i},\omega\>+\<\eta_{i},\nabla^{*}_{f}\nabla\omega\>-2\<\nabla\eta_{i},\nabla\omega\>\\
=&\<\Delta_{f}^{[1]}\eta_{i},\omega\>-2\mathrm{Ric}_{f}(E_{i},\xi)+\<\Delta_{f}^{[1]}\omega,\eta_{i}\>-2\<\nabla\eta_{i},\nabla\omega\>.
\end{align*}
Noting that
\begin{align*}
\Delta_{f}^{[1]}E_{i}=&\Delta_{f}^{[1]}\nabla\<x, \bar{E}_{i}\>=\nabla\Delta_{f}\<x, \bar{E}_{i}\>=0,\\
\mathrm{Ric}_{f}=&-\<A\cdot, A\cdot\>,
\end{align*}
we have that
\begin{equation}\label{5b}
\ \Delta_{f}\<E_{i}, \xi\>=2\<AE_{i},A\xi\>+\<\Delta_{f}^{[1]}\omega,\eta_{i}\>-2\<\nabla\eta_{i},\nabla\omega\>
\end{equation}
Finally, using \eqref{1}, we get
\[
\ \<\nabla \eta_{i},\nabla \omega\>=\<\nabla E_{i},\nabla \xi\>=\sum_{j}\<\nabla_{e_{j}}E_{i},\nabla_{e_{j}}\xi\>=\<\bar{E}_{i},\nu\>\sum_{j}\<Ae_{j},\nabla_{e_{j}}\xi\>.
\]
Substituting in \eqref{5b} we get \eqref{5}.

As for \eqref{6} we can now compute, using \eqref{2}, \eqref{4}, \eqref{5}, that 
\begin{align*}
\Delta_{f}\left(\<\bar{E}_{j}, \nu\>\<E_{i},\xi\>\right)=&\<E_{i},\xi\>\Delta_{f}\<\bar{E}_{j},\nu\>-2\<\nabla \< E_{i}, \xi\>,\nabla \<\bar{E}_{j},\nu\>\>+\<\bar{E}_{j},\nu\>\Delta_{f}\<E_{i},\xi\>\\
=&+|A|^2\<\bar{E}_{i},\nu\>\<E_{i},\xi\>+2\<\nabla\<E_{i},\xi\>,AE_{j}\>+2\<AE_{i},A\xi\>\<\bar{E}_{j},\nu\>\\
&+\<\Delta_{f}^{[1]}\omega,\eta_{i}\>\<\bar{E}_{j},\nu\>-2\<\bar{E}_{j},\nu\>\sum_{k}\<\bar{E}_{i},\nu\>\<Ae_{k},\nabla\xi(e_{k})\>.
\end{align*}

Using \eqref{6}, we obtain that
\begin{align*}
\Delta_{f}\<X_{ij},\xi\>&=|A|^2\<X_{ij},\xi\>+2\left(\<\nabla\<E_{j},\xi\>,AE_{i}\>-\<\nabla\<E_{i},\xi\>,AE_{j}\>\right)\\
&+2\<\bar{E}_{i},\nu\>\<AE_{j},A\xi\>-2\<\bar{E}_{j},\nu\>\<AE_{i},A\xi\>+\<\Delta_{f}^{[1]}\omega,X_{ij}^{\flat}\>.
\end{align*}
Since, by \eqref{1}, 
\begin{align*}
\<\nabla\<E_{j},\xi\>,AE_{i}\>=&\<\nabla_{AE_{i}}E_{j},\xi\>+\<E_{j},\nabla_{AE_{i}}\xi\>\\
=&-\<\bar{E}_{j},\nu\>\<AE_{i},A\xi\>+\<E_{j},\nabla_{AE_{i}}\xi\>,
\end{align*}
we hence get
\begin{equation*}
\Delta_{f}\<X_{ij},\xi\>=|A|^2\<X_{ij},\xi\>+2\left(\<E_{j},\nabla_{AE_{i}}\xi\>-\<E_{i},\nabla_{AE_{j}}\xi\>\right)+\<\Delta_{f}^{[1]}\omega, X_{ij}^{\flat}\>,
\end{equation*}
i.e., \eqref{7}.
\end{proof}

%%%%%%%%%%%%%%%%%%%%%%%%%%%%%%%%%%%%%%%%%%%%%%%%
\section{Proof of the index plus nullity estimate}\label{SecIndPlusNull}
In this section, square brackets $\left[\cdot\right]$ denote the weighted integrals $\left[h\right]=\int_{\Sigma}h\,e^{-f}d\mathrm{vol}_{\Sigma}$. Moreover, we will denote by $L^{2}(\Sigma_{f})$ and $W^{1,2}(\Sigma_{f})$ the Hilbert spaces corresponding to the weighted measure $e^{-f}d\mathrm{vol}_{\Sigma}$ on $\Sigma$.

We start by proving the following
\begin{lemma}\label{EigenfW12}
Let $x:\Sigma^m\to\real{m+1}$ be a translator with $|A|\in L^{\infty}(\Sigma)$ and $\mathrm{Ind}_{f}(\Sigma)=I<+\infty$. Let $\varphi_{1},\ldots,\varphi_{I}$ be orthogonal eigenfunctions of $L_{f}$ in $L^{2}(\Sigma_{f})$ with negative eigenvalue. Then $\varphi_{j}\in W^{1,2}(\Sigma_{f})$.
\end{lemma}
\begin{proof}
Since $\mathrm{Ind}_{f}(\Sigma)=I$ we may find $\varphi_{1},\ldots,\varphi_{I}$ orthogonal eigenfunctions of $L_{f}$ in $L^{2}(\Sigma_{f})$ with negative eigenvalue. Hence, in particular, we have
\[
\ \Delta_{f}\varphi_{j}=|A|^2\varphi_{j}+\lambda_{j}\varphi_{j},
\]
for $j=1,\ldots,I$. Fix an origin $o\in\Sigma$ and denote by $B_{t}$ the geodesic ball in $\Sigma$ of radius $t$ centered at $o$. Let $\eta$ be the cut off-function holding $1$ on $B_{R}$, vanishing on $\Sigma\setminus B_{2R}$ and decaying linearly in between, and denote by $\mathrm{div}_{f}\cdot=e^{f}\mathrm{div}(e^{-f}\cdot)$ the $f$-divergence operator, acting on vector fields on $\Sigma$. Then we have that
\begin{align*}
\mathrm{div}_{f}\left(\eta^{2}\varphi_{j}\nabla\varphi_{j}\right)&=-\eta^2\left(|A|^2+\lambda_{j}\right)\varphi_{j}^2+\<\nabla(\eta^2\varphi_{j}),\nabla\varphi_{j}\>\\
&=-\eta^2\left(|A|^2+\lambda_{j}\right)\varphi_{j}^2+\eta^2|\nabla\varphi_{j}|^2+2\eta\varphi_{j}\<\nabla\eta,\nabla\varphi_{j}\>\\
&\geq-\eta^2\left(|A|^2+\lambda_{j}\right)\varphi_{j}^2+\eta^2|\nabla\varphi_{j}|^2-\varepsilon\eta^2|\nabla\varphi_{j}|^2-\frac{1}{\varepsilon}\varphi_{j}^2|\nabla \eta|^2,
\end{align*}
for every $\varepsilon>0$, where in the last inequality we are using Young's inequality. Hence, choosing $\varepsilon=\frac{1}{2}$, by the $f$-divergence theorem we get
\[
\ \left[\eta^2|\nabla\varphi_{j}|^2\right]\leq2\left[\eta^2\left(|A|^2+\lambda_{j}\right)\varphi_{j}^2\right]+4\left[|\nabla\eta|^2\varphi_{j}^2\right].
\]
By the definition of $\eta$ and the fact that $|A|\in L^{\infty}(\Sigma)$, we get
\begin{equation*}
\int_{B_{R}}|\nabla\varphi_{j}|^2e^{-f}d\mathrm{vol}_{\Sigma}\leq\int_{B_{R+1}}\eta^2|\nabla\varphi_{j}|^2e^{-f}d\mathrm{vol}_{\Sigma}\leq D\left[\varphi_{j}^2\right]<+\infty.
\end{equation*} 
Thus, letting $R\to\infty$, we obtain the desired conclusion.
\end{proof}

\begin{lemma}\label{gW12}
If $|A|\in L^{\infty}(\Sigma)$ and $\omega\in \mathcal{H}^{1}_{f}(\Sigma)$ then $g_{\omega,ij}\doteq\<\omega^{\sharp}, X_{ij}\>\in W^{1,2}(\Sigma_{f})$.
\end{lemma}

\begin{proof}
By Lemma \ref{LemBasEq} we have that
\[
\ \Delta_{f}g_{\omega, ij}=|A|^2
g_{\omega, ij}+2v_{\omega, ij}.
\]
Hence we have that
\begin{align*}
\left[-|A|^2\eta^2g_{\omega, ij}^2-2v_{\omega, ij}\eta^2g_{\omega, ij}\right]=&\left[-\eta^2g_{\omega, ij}^2\Delta_{f}g_{\omega, ij}\right]\\
=&\left[\mathrm{div}_{f}\left(\eta^2g_{\omega, ij}\nabla g_{\omega, ij}\right)\right]-\left[\eta^2|\nabla g_{\omega, ij}|^2\right]-2\left[\eta g_{\omega, ij}\<\nabla \eta, \nabla g_{\omega, ij}\>\right].
\end{align*}
Let $\eta$ be the cut-off function holding $1$ on $B_{R}$, vanishing on $\Sigma\setminus B_{2R}$ and decaying linearly in between. By means of Young's inequality we get that
\begin{align*}
\left[\eta^2|\nabla g_{\omega, ij}|^2\right]=&\left[|A|^2\eta^2g_{\omega, ij}^2\right]+2\left[v_{\omega, ij}\eta^2g_{\omega, ij}\right]-2\left[\eta g_{\omega, ij} \<\nabla \eta, \nabla g_{\omega, ij}\>\right]\\
\leq&\left[\eta^2|A|^2g_{\omega, ij}^2\right]+2\left[v_{\omega, ij}\eta^2g_{\omega, ij}\right]+\varepsilon\left[\eta^2|\nabla g_{\omega, ij}|^2\right]+\frac{1}{\varepsilon}\left[|\nabla \eta|^2 g_{\omega, ij}^2\right],
\end{align*}
for any $\varepsilon>0$. Choose $\varepsilon=\frac{1}{2}$, letting $|A|\leq C$ for some constant $C>0$, we hence obtain
\begin{equation*}
\left[\eta^2|\nabla g_{\omega, ij}|^2\right]\leq 2C\left[\eta^2 g_{\omega, ij}^2\right]+4\left[v_{\omega, ij}\eta^2g_{\omega, ij}\right]+4\left[|\nabla\eta|^2g_{\omega, ij}^2\right].
\end{equation*}
Note now that one can readily compute that
\begin{align*}
\sum_{i,j=1}^{m+1}\eta^{2}g_{\omega, ij}^2=&2\eta^2|\omega|^2,\\
\sum_{i,j=1}^{m+1}|\nabla\eta|^2g_{\omega, ij}^2=&2|\nabla\eta|^2|\omega|^2,\\
\sum_{i,j=1}^{m+1}g_{\omega, ij}v_{\omega, ij}=&0.
\end{align*}
We thus deduce that
\begin{equation*}
\left[\eta^2|\nabla g_{\omega, ij}|^2\right]\leq\sum_{i,j=1}^{m+1}\left[\eta^2|\nabla g_{\omega, ij}|^2\right]\leq 4C\left[\eta^2|\omega|^2\right]+8\left[|\nabla \eta|^2|\omega|^2\right].
\end{equation*}
Hence
\[
\int_{B_{R}}|\nabla g_{\omega, ij}|^2e^{-f}d\mathrm{vol}_{\Sigma}\leq\left[\eta^2|\nabla g_{\omega, ij}|^2\right]\leq D\left[|\omega|^2\right]<+\infty,
\]
with $D>0$ constant. By the dominated convergence Theorem we conclude that $g_{\omega, ij}\in W^{1,2}(\Sigma_{f})$.
\end{proof}

Reasoning as in Proposition 2.4 of \cite{ChaoLi} one can prove the following

\begin{proposition}\label{fPropChaoLi}
Let $x:\Sigma^m\to\real{m+1}$ be a translator with $|A|\in L^{\infty}(\Sigma)$ and $\mathrm{Ind}_{f}(\Sigma)=I<+\infty$, and let $\varphi_{1},\ldots,\varphi_{I}$ be $I$ eigenfunctions of $L_{f}$ corresponding to negative eigenvalues given in Lemma \ref{EigenfW12}. Then for any function $h\in C^{\infty}(\Sigma)\cap W^{1,2}(\Sigma_{f})$ which is $L^{2,f}$ orthogonal to $\varphi_{1}, \ldots, \varphi_{I}$, we have that $Q_{f}(h,h)\geq 0$. Moreover if $Q_{f}(h,h)=0$ then $h$ is a solution of $L_{f}h=0$.
\end{proposition}

We can now provide the proof of our index plus nullity estimate.

\begin{proof}[Proof of Theorem \ref{ThIndPlusNullEst}] Since we are assuming that $\mathrm{Ind}_{f}(\Sigma)=I<+\infty$, by \cite[Section 4]{IR_fMin}, we know that $\mathcal{V}\doteq\mathcal{H}^{1}_{f}(\Sigma)$ is finite dimensional. By Lemma \ref{EigenfW12} there exist $\varphi_{1},\ldots,\varphi_{I}$ smooth $W^{1,2}(\Sigma_{f})$ eigenfunctions of $L_{f}$ corresponding to negative eigenvalues.  Consider the linear map $F:\mathcal{V}\to\real{{{m+1}\choose{2}}I}$ given by
\[
\ F(\omega)=\left([g_{\omega,ij}\varphi_{l}]\right),
\]
with $1\leq i,j\leq m+1$ and $1\leq l\leq I$. Assume that $k\doteq\mathrm{dim}\mathcal{V}>{{m+1}\choose{2}}I$ (if this is not the case we would get a better estimate than our claim). Then there exist at least $k-{{m+1}\choose{2}}I$ linearly independent $f$-harmonic $1$-forms $\omega\in\mathrm{Ker}\,F$. Thus, by Proposition 2 in \cite{ChengZhou_Index} and Lemma \ref{LemBasEq} we get
\begin{equation*}
0\leq \left[g_{\omega,ij}L_{f}g_{\omega,ij}\right]=Q_{f}(g_{\omega,ij},g_{\omega,ij})=2\left[g_{\omega,ij}v_{\omega,ij}\right].
\end{equation*}
We note now that one can prove directly, using the definitions, that
\begin{equation*}
\sum_{i,j=1}^{m+1}\left[g_{\omega,ij}v_{\omega,ij}\right]=0.
\end{equation*}
Hence
\[
\ \sum_{i,j=1}^{m+1}Q_{f}(g_{\omega,ij},g_{\omega,ij})=0,
\]
and, since each of these terms is nonnegative, we get that $g_{\omega,ij}\in \mathrm{Ker}\,Q_{f}$ for every $1\leq i,j \leq m+1$. Since by Lemma \ref{gW12} we have that $g_{\omega, ij}\in W^{1,2}(\Sigma_{f})$, by Proposition \ref{fPropChaoLi} we obtain that $g_{\omega,ij}\in \mathrm{Ker}\,L_{f}$. To conclude the proof we have to show that the $k-{{m+1}\choose{2}}I$ linearly independent $f$-harmonic forms $\omega$ generate at least $\frac{2}{m(m+1)}k-I$ linearly independent functions $g_{\omega,ij}$. This is the content of the next lemma. Thus
\[
\ \mathrm{Null}_{f}(\Sigma)\geq\frac{2}{m(m+1)}k-I,
\] 
and \eqref{IndEst} follows. Remark \ref{Rem1.1} yields the second part of the theorem.
\end{proof}

By minor modifications to the proof of Proposition 4.3 in \cite{ChaoLi} one can prove the  following

\begin{lemma}
Let $\mathcal{H}$ be an $h$-dimensional subspace of $\mathcal{H}^{1}_{f}(\Sigma)$. Then the set 
\[
\ \left\{g_{\omega, ij}\,:\,\omega\in\mathcal{H},\quad 1\leq i,j\leq m+1\right\}
\]
has at least $\frac{2}{m(m+1)}h$ linearly independent $L^{2}(\Sigma_{f})$ smooth functions on $\Sigma$.
\end{lemma}
%%%%%%%%%%%%%%%%%%%%%%%%%%%%%%%%%%%%%%%%%%%%%%
\section{Estimating the nullity term in some special situations}\label{EstNull}

When there exists a point where all principal curvatures are distinct we can estimate the nullity of the stability operator. Theorem \ref{ThIndEst} easily follows from Equation \eqref{7}, the proof of Theorem \ref{ThIndPlusNullEst}, and the following

\begin{proposition}
Let $x:\Sigma^m\rightarrow \real{m+1}$ be a translator. Suppose that at one point $p$ on $\Sigma$ all principal curvatures of $\Sigma$ are different. Then the dimension of the function space 
\[
\mathcal{W}:=\{\omega\in \mathcal{H}^{1}_{f}(\Sigma)  \ :\ \nabla \omega(AX,Y)=\nabla \omega(AY,X), \,\forall\, X,Y\in T\Sigma\}
\]
is at most $2m-1$.
\end{proposition}
\begin{proof}
Let $x:\Sigma^m \to \real{m+1}$ be a translator and let $f(p)=-\left\langle p, \overline{E}_{m+1}\right\rangle$, $p\in\real{m+1}$. We choose an orthonormal frame $\left\{e_{i}\right\}$ on $\Sigma$ so that in a small neighbourhood $U_{p}$ of the point $p$ the $e_i$'s are principal directions with corresponding (all distinct) principal curvatures $\lambda_i$, $1\leq i\leq m$. Then, for any $\omega \in \mathcal{W}$, we get $\nabla \omega (e_i,e_j)=0$ for $i\neq j$, $i,j=1,\dots,m$. Moreover, since $\delta_{f}\omega=0$, we have that
\begin{equation}\label{RelTromega}
\nabla\omega(e_{m},e_{m})=-\sum_{i=1}^{m-1}\nabla\omega(e_{i},e_{i})+\omega(\nabla f)=-\sum_{i=1}^{m-1}\nabla\omega(e_{i},e_{i})+\sum_{i=1}^{m}\<\nabla f,e_{i}\>\omega(e_{i}).
\end{equation}
Hence $\omega(e_{1}),\ldots, \omega(e_{m}), \nabla\omega(e_{1}, e_{1}), \ldots,\nabla\omega(e_{m-1},e_{m-1})$ completely determine the tensor $\nabla \omega$ on $U_{p}$. 

Let us now consider the functions defined on  $U_p$ by
\[
\phi_i(q)=
\begin{cases}
\omega(e_i)(q),\quad \textrm{for }1 \leq i \leq m\\
\nabla\omega(e_{i-m},e_{i-m})(q),\quad\textrm{for }m+1\leq i \leq 2m.
\end{cases}
\]
Keeping in mind Equation \eqref{RelTromega}, we can reason, with minor modifications, as in \cite[Proposition 5]{AmbrozioCarlottoSharp_PAMS} and prove that if $\omega\in \mathcal{W}$ then, for every $q \in U_p$, the values of $(\phi_1,\dots, \phi_{2m})(q)$ are uniquely determined by $(\phi_1,\dots, \phi_{2m-1})(p)$. Hence the space $\mathcal{W}|_{U_p}$ has dimension at most $2m-1$ and the general statement over $\Sigma$ follows by unique continuation.
\end{proof}

%%%%%%%%%%%%%%%%%%%%%%%%%%%%%%%%%%%%%%%%%%%%%%%%
\section{Relating $f$-harmonic $1$-forms and topology}\label{SecRelfHarmTop}

Recall that, given a Riemannian manifold $\Sigma$, we can define the cohomology with compact support  $H_c^k(\Sigma)$ as follows. Consider the exact sequence:
\[
\cdots \longrightarrow C^{\infty}_c(\Lambda^{k-1}T^*\Sigma)\longrightarrow C^{\infty}_c(\Lambda^{k}T^*\Sigma)\longrightarrow C^{\infty}_c(\Lambda^{k+1}T^*\Sigma)\longrightarrow\cdots
\] 
Then we can define
\[
H^k_c(\Sigma)=\mathrm{ker}\{d C^{\infty}_c(\Lambda^{k}T^*\Sigma)\rightarrow C^{\infty}_c(\Lambda^{k+1}T^*\Sigma)/dC^{\infty}_c(\Lambda^{k-1}T^*\Sigma)\}.
\]
In presence of a weighted $L^2$-Sobolev inequality, one can relate the space $\mathcal{H}^{1}_{f}(\Sigma)$ to the space $H^1_c(\Sigma)$. Indeed we have the validity of the following proposition inspired by \cite{Carron_Sem}; compare with \cite[Proposition 3.2]{KS}.

\begin{proposition}
Let $\Sigma^m_f$ be a complete weighted manifold satisfying, for some $0\leq\alpha<1$, the weighted $L^2$-Sobolev inequality
\begin{equation}\label{eq:weightedSobolev}
\Big(\int_\Sigma u^\frac{2}{1-\alpha}e^{-f}d\mathrm{vol}_\Sigma\Big)^{1-\alpha}\leq S(\alpha)^2 \int_\Sigma|\nabla u|^2 e^{-f}d\mathrm{vol}_\Sigma,
\end{equation}
for some positive constant $S(\alpha)$ and for every $u\in C^\infty_c(\Sigma)$. Then
\[
\mathrm{dim}(\mathcal{H}^{1}_{f}(\Sigma))\geq \mathrm{dim}(H_c^1(\Sigma)).
\]
\end{proposition}
\begin{proof}
By \cite{Bue}, we have the following decomposition of the space $L^{2,f}(\Lambda^1T^*\Sigma)$ of square integrable one forms on $\Sigma$ with respect to the weighted measure:
\[
L^{2,f}(\Lambda^1T^*\Sigma)=A\oplus B_f \oplus \mathcal{H}^{1}_{f}(\Sigma),
\]
where
\[
\begin{cases}
A=\overline{\{dg: g\in C^\infty_c(\Sigma)\}},\\
B_f=\overline{\{\delta_f\eta: \eta\in C^{\infty}_c(\Lambda^2T^*\Sigma)\}},
\end{cases}
\]
the closure been taken with respect to the weighted $L^2$ norm.
Denote by $Z^1_f(\Sigma)$ the space of closed $L^{2,f}$ $1$-forms, that is
\[
Z^1_f(\Sigma)=\{\omega \in L^{2,f}(\Lambda^1T^*\Sigma): d\omega=0\}.
\]
It is understood that the equation $d\omega=0$ holds weakly on $\Sigma_f$, i.e.
\[
\langle \omega, \delta_f \beta \rangle=0 \ \forall \beta\in C^{\infty}_c(\Lambda^2T^*\Sigma).
\]
In particular,
\[
Z^1_f(\Sigma)=A\oplus \mathcal{H}^{1}_{f}(\Sigma),
\]
and hence the first space of reduced $L^{2,f}$ cohomology satisfies
\[
H^1_f(\Sigma)\doteq\frac{Z^1_f(\Sigma)}{A}\simeq \mathcal{H}^{1}_{f}(\Sigma).
\]
Now observe that there is a natural map
\[
i:H^1_c(\Sigma)\rightarrow H^1_f(\Sigma).
\]
We claim that this map is injective. This is equivalent to prove that if $\alpha$ is a closed $1$-form which is zero with respect to the $L^{2,f}$ cohomology, then there exists $u\in C_{c}^{\infty}(\Sigma)$ such that $\alpha=du$. First note that if $\alpha$ is zero with respect to the $L^{2,f}$ cohomology then there exists a sequence $u_j\in C^{\infty}_c(\Sigma)$ such that
\[
\lim_{j\rightarrow+\infty}\|\alpha-d u_j\|=0. 
\]
The validity of the weighted $L^2$-Sobolev inequality implies that $\{u_j\}$ is a Cauchy sequence in $L^\frac{2}{1-\alpha}(\Sigma_f)$ and hence it will converge to a function $u$ in this space. Moreover an adaptation to the weighted setting of \cite[Lemma 1.11]{Carron_LectNotes} yields that  $u$ must satisfy $du=\alpha$. Since $\alpha$ has compact support, $u$ must be locally constant at infinity. Since $u\in L^\frac{2}{1-\alpha}(\Sigma_f)$ and each end of $\Sigma$ has infinite $f$-volume, this implies that $u$ must have compact support, proving the claim.

To conclude note that the injectivity of the map $i$ implies that
\[
\mathrm{dim}(H_c^1(\Sigma))=\mathrm{rank}(i)\leq \mathrm{dim}(\mathcal{H}^{1}_{f}(\Sigma)),
\]
proving the proposition.
\end{proof}

Some remarks are in order
\begin{remark}
\rm{If $\Sigma$ is topologically tamed, i.e. it is diffeomorphic to the interior of a compact manifold $\bar{\Sigma}$ with compact boundary $\partial \bar{\Sigma}$, then $H_c^1(\Sigma)$ is isomorphic to the relative cohomology group of $\bar{\Sigma}$:
\[
H^1(\bar{\Sigma},\partial \bar{\Sigma}):=\frac{\{\alpha \in C^\infty(\Lambda^1T^*\bar{\Sigma}), d\alpha=0, \iota^*\alpha=0\}}{\{d\beta,\beta\in C^\infty(\bar{\Sigma}), \iota^*\beta=0\}},
\]
where $\iota:\partial \bar{\Sigma}\rightarrow \bar{\Sigma}$ is the inclusion map.
Moreover, setting $K=\partial \bar{\Sigma}$, as a consequence of the exactness of the sequence:
\[
0\longrightarrow H^0(\bar{\Sigma})\longrightarrow H^0(K)\longrightarrow H^1(\bar{\Sigma}, K)\longrightarrow H^1(\bar{\Sigma})\longrightarrow H^1(K)\longrightarrow\cdots
\]
we also deduce that
\[
\mathrm{dim}(H_c^1(\Sigma))\geq b_1(\bar{\Sigma})+\sharp\{\textrm{connected components of\,} K\}-1-b_1(K).
\]
Finally, we also note that if $\Sigma$ is of finite topological type, i.e it is diffeomorphic to $\bar{\Sigma}\backslash K$ with $K=\{p_1,\dots,p_r\}$, then
\[
\mathrm{dim}(H_c^1(\Sigma))= b_1(\bar{\Sigma})+r-1.
\]}
\end{remark}

Since, as a consequence of Appendix \ref{AppA} and arguments  in \cite{ImperaRimoldi_Transl}, on every translator with $m\geq 3$ and which is contained in a upper halfspace we have the validity of \eqref{eq:weightedSobolev} with the choice $\alpha=2/m$, it is straightforward, keeping in mind the previous remarks, to obtain the following
\begin{lemma}\label{lemmacohom}
Let $x:\Sigma^{m\geq 3}\rightarrow\real{m+1}$ be a translator such that $x(\Sigma)$ is contained in a upper halfspace. Then
\[
\mathrm{dim}(\mathcal{H}^{1}_{f}(\Sigma))\geq \mathrm{dim}(H_c^1(\Sigma)).
\]
In particular, if $\Sigma$ has finite topological type then 
\[
\mathrm{dim}(\mathcal{H}^{1}_{f}(\Sigma))\geq b_1(\bar{\Sigma})+r-1.
\]
\end{lemma}

From Lemma \ref{lemmacohom} and Theorem \ref{ThIndPlusNullEst} we obtain the validity of Theorem \ref{ThFiniteTop} in the Introduction.

\begin{remark}
\rm{We end up this section noting that if $x :\Sigma^m\rightarrow\real{m+1}$ is a $f$-stable translator with $|A|\in L^{\infty}(\Sigma)$ then $\mathrm{Null}_{f}(\Sigma)=1$. Indeed, letting $h\in C^2(\Sigma)$ be a positive solution of the stability equation
\[
\Delta_f h - |A|^2h = 0,
\]
and letting $u\in L^2(\Sigma_f)$ be a solution of 
\[
\Delta_f u-|A|^2u=0,
\]
then $u=C\omega$ for some non-zero constant $C$. Once noted that, since $|A|\in L^{\infty}(\Sigma)$,  $u\in W^{1,2}(\Sigma_f)$ the claim can be obtained proceeding exactly as in \cite[Lemma 3.2]{ImperaRimoldi_Transl} replacing $H$ by $u$. In particular, by Theorem \ref{ThIndPlusNullEst}, this yields (yet another) alternative proof of the fact that $f$-stable $2$-dimensional translators with $|A|\in L^{\infty}(\Sigma)$ have at most genus one.}
\end{remark}

\appendix
\section{About the validity of a  weighted Sobolev inequality on translators}\label{AppA}

There is a gap in the proof of Lemma 4.1 and Lemma 4.2 in \cite{ImperaRimoldi_Transl}. Here we provide a correct complete proof of Lemma 4.2 in \cite{ImperaRimoldi_Transl}.

\begin{lemma}[Lemma 4.2 in \cite{ImperaRimoldi_Transl}]\label{4.2}
Let $x:\Sigma^{m}\to\mathbb{R}^{m+1}$ be a translator contained in the upper halfspace $\Pi_{\bar{V},a}=\left\{p\in\mathbb{R}^{m+1}:\,\<p,\bar{V}\>\geq a\right\}$, for some $a\in\mathbb{R}$. Let $h$ be a non-negative compactly supported $C^{1}$ function on $\Sigma$. Then
\[
\ \left[\int_{\Sigma}h^{\frac{m}{m-1}}e^{-f}d\mathrm{vol}_{\Sigma}\right]^{\frac{m-1}{m}}\leq D\int_{\Sigma}|\nabla h|e^{-f}d\mathrm{vol}_{\Sigma},
\]
for some constant $D$ depending on $a$ and $m$.
\end{lemma}
\begin{proof}
Assume that $x:\Sigma^{m}\to\mathbb{R}^{m+1}$ is a translator for the mean curvature flow contained in the upper halfspace $\Pi_{\overline{V},a}$ and let $f(p)=-\<p, \overline{V}\>$. Consider on $\mathbb{R}^{m+1}$ the conformal metric $\widetilde{\<\,,\,\>}\doteq e^{-\frac{2f}{m}}\<\,,\,\>$. Then $(\Sigma, x^{*}\<\,,\,\>)$ is a translator in $(\Pi_{v,a},\<\,,\,\>)$ if and only if $\tilde{\Sigma}\doteq(\Sigma, \tilde{g}\doteq x^{*}\widetilde{\<\,,\,\>})$ is minimal in $(\Pi_{\bar{V},a},\widetilde{\<\,,\,\>})$. Without loss of generality we can assume that $x(\Sigma)$ does not intersect the boundary of $\Pi_{\bar{V},a}$.

Using the expression for the curvature tensor of a Riemannian manifold under a conformal change one proves that the curvature tensor $\tilde{R}$ of $(\Pi_{\bar{V},a},\widetilde{\<\,,\,\>})$ satisfies 
\[
\tilde{R}_{ijij}=\frac{e^{-\frac{2f}{m}}}{m^2}\left(\<\bar{\nabla} f, e_{i}\>^{2}+\<\bar{\nabla}f, e_{j}\>^2-|\bar{\nabla}f|^2\right)\leq 0,
\]
where $\left\{\partial_{j}\doteq\frac{\partial}{\partial x^{j}}\right\}_{j=1}^{m+1}$ denotes the standard orthonormal basis in $\mathbb{R}^{m+1}$.

Since $\Pi_{\bar{V},a}$ is a manifold with boundary we have to pay special attention in applying directly Theorem 2.1 in \cite{HS}. However, it is possible to extend $(\Pi_{\bar{V},a}, \widetilde{\langle\,,\,\rangle})$ to a complete simply connected manifold without boundary preserving the curvature bound. Indeed, let $\psi:\mathbb{R}\to\mathbb{R}$ be a smooth convex function such that $\psi(t)=\frac{2}{m}t$ for $t\geq a$, and $\psi(t)>1$ for $t\leq a-C$ for some constant $C>0$, and define $\overline{\langle\,,\,\rangle}\doteq e^{\psi(\langle p,\bar{V}\rangle)}\langle\,,\,\rangle$. Then  $\bar{N} =\left(\mathbb{R}^{m+1}, \overline{\langle\,,\,\rangle}\right)$ has the desired properties. 
Indeed, $\bar{N}$ is simply connected since it is topologically the Euclidean space. By the definition of $\psi$, $\bar{N}$ is a Riemannian extension of $(\Pi_{\bar{V},a}, \widetilde{\langle\,,\,\rangle})$ and, since $\Sigma$ is strictly contained in $\Pi_{\bar{V}, a}$, $\tilde{\Sigma}$ can be viewed as a minimal hypersurface in $\bar{N}$.

Furthermore, in order to prove that $\bar{N}$ is complete, it is sufficient to prove that divergent curves in $\mathbb{R}^{m+1}$ have infinite length with respect to the metric $\overline{\langle\,,\,\rangle}$. In this regard note first that,by the definition of $\psi$, for any $t\in \mathbb{R}$,
\[
\ e^{\psi(t)}\geq \min\left\{e,\,e^{\min_{\left[a-C,a\right]}\psi},\,e^{\frac{2}{m}a}\right\}\doteq \lambda>0.
\]
Hence, if $\gamma:I\to\mathbb{R}^{m+1}$ is any curve, then its length satisfies
\begin{align*}
\bar{l}(\gamma)=&\int_{I}e^{\frac{\psi}{2}\left(\langle\gamma, \bar{V}\rangle\right)}\langle\dot{\gamma},\dot{\gamma}\rangle^{\frac{1}{2}}dt\\
\geq& \sqrt{\lambda}\int_{I}\langle\dot{\gamma},\dot{\gamma}\rangle^{\frac{1}{2}}dt=\sqrt{\lambda}\,l_{Eucl}(\gamma).
\end{align*}
In particular if $\gamma$ is a divergent curve, since we know by the completeness of the Euclidean space endowed with the standard metric that $l_{Eucl}(\gamma)$ is infinite, we conlcude that $\bar{l}(\gamma)$ is infinite as well.

Moreover, for every $i\neq j$, we have that
\begin{equation*}\label{ConfChange}
\ \bar{R}_{ijij}=e^{\psi(-f)}\left[-\langle\,,\,\rangle\mathbin{\bigcirc\mspace{-15mu}\wedge\mspace{3mu}}\left(\mathrm{Hess}(\frac{\psi}{2}(-f))-d(\frac{\psi}{2}(-f))\otimes d(\frac{\psi}{2}(-f))+\frac{1}{2}\left|d(\frac{\psi}{2}(-f))\right|^2\langle\,,\,\rangle\right)\right]_{ijij}.
\end{equation*} 
Using the expression of $f$,
\begin{align*}
\alpha\doteq&\mathrm{Hess}(\frac{\psi}{2}(-f))-d(\frac{\psi}{2}(-f))\otimes d(\frac{\psi}{2}(-f))+\frac{1}{2}\left|d(\frac{\psi}{2}(-f))\right|^2\langle\,,\,\rangle\\
=&\left(\frac{1}{2}\psi^{\prime\prime}(\langle p,\bar{E}_{m+1}\rangle)-\frac{1}{4}(\psi^{\prime}(\langle p,\bar{E}_{m+1}\rangle))^2\right)dx^{m+1}\otimes dx^{m+1}\\
&+\frac{1}{8}(\psi^{\prime}(\langle p,\bar{E}_{m+1}\rangle))^2\langle\,,\,\rangle.
\end{align*}
Hence
\begin{align*}
\left(\langle\,,\,\rangle\mathbin{\bigcirc\mspace{-15mu}\wedge\mspace{3mu}}\alpha\right)_{ijij}=&\alpha_{jj}+\alpha_{ii}-2\delta_{ij}\alpha_{ij}\\
=&\left(\frac{1}{2}\psi^{\prime\prime}(\langle p,\bar{E}_{m+1}\rangle)-\frac{1}{4}\left(\psi^{\prime}(\langle p,\bar{E}_{m+1}\rangle)\right)^{2}\right)\\&\left[\left(dx^{m+1}\left(\partial_{j}\right)\right)^2+\left(dx^{m+1}\left(\partial_{i}\right)\right)^{2}-2\delta_{ij}dx^{m+1}\left(\partial_{i}\right)dx^{m+1}\left(\partial_{j}\right)\right]\\
&+\frac{1}{8}\left(\psi^{\prime}(\langle p,\bar{E}_{m+1}\rangle)\right)^{2}\left(2-2\delta_{ij}.\right)
\end{align*}
Hence $\bar{R}_{ijij}\leq 0$ for any $i\neq j$. 
\medskip

Note that, since $\bar{N}$ is complete, simply connected and with non-positive sectional curvature, as a consequence of the Cartan-Hadamard theorem, its injectivity radius is infinite.

By \cite{HS} we can now get on the minimal hypersurface $\tilde{\Sigma}$ of $\bar{N}$ the validity of the $L^{1}$- Sobolev inequality
\begin{equation}\label{Sob1}
\left(\int_{\Sigma}h^{\frac{m}{m-1}}d\mathrm{vol}_{\tilde{\Sigma}}\right)^{\frac{m-1}{m}}\leq C\int_{\Sigma}\widetilde{|\tilde{\nabla}h|}d\mathrm{vol}_{\tilde{\Sigma}},
\end{equation}
for every non-negative $h\in C_{c}^{1}(\Sigma)$.

The desired conclusion follows immediately from \eqref{Sob1} keeping in mind that under a conformal change of the metric, the volume form and the norm of the gradient of a given function satisfy
\begin{eqnarray*}
d\mathrm{vol}_{\tilde{\Sigma}}=e^{-f}d\mathrm{vol}_{\Sigma},\\
\widetilde{|\tilde{\nabla}h|}=e^{\frac{f}{m}}|\nabla h|\leq e^{\frac{a}{m}}|\nabla h|.
\end{eqnarray*}
\end{proof}

\begin{acknowledgement*}
The authors are deeply grateful to Alessandro Savo for his interest in this work and a number of enlightening discussions. The first author is partially supported by INdAM-GNSAGA. The second author acknowledge partial support by INdAM-GNAMPA.
\end{acknowledgement*}

\end{document}